\documentclass[11pt]{amsart}

\usepackage[utf8]{inputenc}

\usepackage[backend=biber,style=alphabetic,maxbibnames=50]{biblatex}
\usepackage{amssymb,amsfonts}
\usepackage{amsmath,amscd}
\usepackage[all,arc]{xy}
\usepackage{enumerate}
\usepackage{mathrsfs}
\usepackage[pdftex]{graphicx}
\usepackage[T1]{fontenc}
\usepackage[usenames,dvipsnames]{color}
\usepackage[bookmarks=false]{hyperref}
\usepackage{dsfont}
\usepackage{tikz-cd}
\usetikzlibrary{automata,positioning}
\usepackage{fullpage}
\usepackage{enumitem}

\usepackage{color}
\usepackage{xcolor}

\theoremstyle{plain}
\newtheorem{thm}{Theorem}[section]

\newtheorem{lem}[thm]{Lemma}

\theoremstyle{definition}
\newtheorem{defn}[thm]{Definition}

\newtheorem{aDD^+m}[thm]{ADD^+endum}

\theoremstyle{remark}
\newtheorem{rmk}[thm]{Remark}

\newcommand{\bbC}{\mathbb{C}} 

\newcommand{\bbP}{\mathbb{P}}

\newcommand{\PSL}{\text{PSL}}

\newcommand*{\defeq}{\mathrel{\vcenter{\baselineskip0.5ex \lineskiplimit0pt
			\hbox{\scriptsize.}\hbox{\scriptsize.}}}%
	=}

\DeclareMathOperator{\VD}{VD}
\DeclareMathOperator{\HD}{HD}

\DeclareMathOperator{\Jac}{Jac}

\title{Manhattan curves in complex dynamics
and \\ asymptotic correlation
of multiplier spectra}

\author{Fabrizio Bianchi}
\address{Dipartimento di Matematica, Università di Pisa, Largo Bruno Pontecorvo 5, 56127 Pisa, Italy}
 \email{fabrizio.bianchi$@$unipi.it}
\author{Yan Mary He}
\address{Department of Mathematics\\
	University of Oklahoma\\
	Norman, OK 73019}
\email{he$@$ou.edu}
\date{\today}

\begin{document}

\maketitle

\begin{center}
\emph{Dedicated to Manfred Denker for 
his 60th birthday.}
\end{center}

\begin{abstract}
The Manhattan curve for a pair of hyperbolic structures (possibly with cusps) on a given surface is a geometric object that encodes the growth rate of lengths of closed geodesics with respect to the two different hyperbolic metrics. It has been extensively studied as a way to understand geodesics on surfaces, the thermodynamic formalism of the geodesic flows and
comparison of hyperbolic metrics.

Via Sullivan's dictionary, in this paper, we define and study
the Manhattan curve for a pair of 
hyperbolic rational
maps on $\mathbb C\mathbb P^1$,
and more generally of
holomorphic endomorphisms of $\mathbb{C}\mathbb{P}^k$.
We discuss several counting results
for the multiplier spectrum 
and show that the Manhattan curve for two holomorphic endomorphisms is related to the correlation number of their multiplier
spectra.
\end{abstract}

\section{Introduction}
If $S$ is a closed surface of genus at least 2, a hyperbolic structure on $S$ can be identified with a discrete and faithful representation $\rho \colon \pi_1S \to \PSL(2,\mathbb{R})$. The image $\Gamma\defeq \rho(\pi_1S) < \PSL(2,\mathbb{R})$ is called a convex cocompact Fuchsian group; namely, it is a discrete subgroup of $\PSL(2,\mathbb{R})$ and the convex hull of the limit set
of $\Gamma$ in $\mathbb{H}^2$ has compact quotient. Moreover, we have $S = \mathbb{H}^2/\Gamma$, where $\mathbb{H}^2$ denotes the hyperbolic plane. See \cite{Nicholls89,KatokS_92} for more details on Fuchsian groups. 

Given a convex cocompact Fuchsian group $\Gamma$, the classical {\it Poincar\'e series} of $\Gamma$ is defined as
$$P_\Gamma(s,x) \defeq \sum_{\gamma \in \Gamma}e^{-sd(x,\gamma x)}$$
for every $s \in \mathbb{R}$
and $x \in \mathbb{H}^2$, where
$d(\cdot, \cdot)$ is the hyperbolic distance on $\mathbb{H}^2$. It is well-known \cite{Nicholls89}
that the Poincar\'e series has a critical exponent $\delta_\Gamma$ for which,
for all $x \in \mathbb{H}^2$,
$P_\Gamma(s,x)$ converges for all
$s > \delta_\Gamma$,  and diverges for all
$s <\delta_\Gamma$.
The critical exponent is also the Hausdorff dimension of the limit set of $\Gamma$ as well as the topological entropy of the geodesic flow on the unit tangent bundle of $S$; see for instance \cite{Sullivan79, Nicholls89}.

If $\rho_1$ and $\rho_2$ are two discrete faithful representations of $\pi_1S$ into $\PSL(2,\mathbb{R})$, for $(a,b) \in \mathbb{R}_{\ge 0}\times \mathbb{R}_{\ge 0}\setminus\{(0,0)\}$,
the {\it weighted Poincar\'e series for $\rho_1$ and $\rho_2$} is defined as
$$P_{\rho_1,\rho_2}^{a,b}(s,x_1,x_2) \defeq \sum_{g \in \pi_1S}e^{-s(a \cdot d(x_1, \rho_1(g) x_1)+b \cdot d(x_2,\rho_2(g) x_2))}$$
for every
$s \in \mathbb{R}$ and $x_1,x_2 \in \mathbb{H}^2$. One can then
denote by $\delta_{\rho_1,\rho_2}^{a,b}$ the critical exponent for $P_{\rho_1,\rho_2}^{a,b}$;
namely,
for every $x_1,x_2 \in \mathbb{H}^2$,
$P_{\rho_1,\rho_2}^{a,b}(s,x_1, x_2)$ converges if $s > \delta_{\rho_1,\rho_2}^{a,b}$
and diverges if $s <\delta_{\rho_1,\rho_2}^{a,b}$.

The {\it Manhattan curve} $\mathcal{C} = \mathcal{C}(\rho_1,\rho_2)$ is defined as
$$\{(a,b) \in \mathbb{R}_{\ge 0}\times \mathbb{R}_{\ge 0}\setminus\{(0,0)\}: \delta_{\rho_1,\rho_2}^{a,b}=1 \}.$$
In \cite{Burger93}, using the
Patterson-Sullivan theory \cite{Patterson76, Sullivan79},
Burger proved that if $\rho_1,\rho_2 \colon \pi_1S \to \PSL(2,\mathbb{R})$ are convex cocompact representations, then the Manhattan curve $\mathcal{C}(\rho_1,\rho_2)$ is $C^1$. This result was strengthened by Sharp \cite{Sharp98}, who 
used thermodynamic formalism to prove that $\mathcal{C}(\rho_1,\rho_2)$ is in fact real-analytic. In \cite{Kao20Manhattan, Kao20}, Kao studied the Manhattan curves for hyperbolic surfaces with cusps and for punctured surfaces. In particular, he proved that the Manhattan curves are real-analytic in these cases as well. Very recently, Kao--Martone \cite{KM24}  investigated Manhattan curves for cusped Hitchin representations. Cantrell--Tananka \cite{CT21}  investigated Manhattan curves associated to a pair of left-invariant hyperbolic metrics on a hyperbolic group.

Given a hyperbolic surface $\mathbb{H}^2/\rho_1(\pi_1S)$ and a conjugacy class $[\gamma]$ in $\pi_1S$, we denote by $l_1([\gamma])$ the hyperbolic length of the unique closed geodesic in the conjugacy class $[\gamma]$ on the surface $\mathbb{H}^2/\rho_1(\pi_1S)$. Pollicott-Sharp \cite{PS98} proved the following asymptotic growth rate for the length spectrum of primitive closed geodesics (i.e., geodesics which are not multiples of another geodesic) on $\mathbb{H}^2/\rho_1(\pi_1S)$: there exists $0<c<h$ 
(where $h$ is the entropy of the geodesic flow of $\mathbb{H}^2/\rho_1(\pi_1S)$)
such that
\begin{equation}\label{eq:intro-1}
{\rm Card} \{[\gamma] : l_1([\gamma]) \le T \} = Li(e^{hT}) +O(e^{cT}),
\quad 
\mbox {where } 
Li(y) = \int_2^y \frac{1}{\log u} du.
\end{equation}

The Manhattan curve $\mathcal{C}(\rho_1,\rho_2)$ is related to the \emph{correlation of 
the length
spectra}
on the hyperbolic surfaces $\mathbb{H}^2/\rho_1(\pi_1S)$ and $\mathbb{H}^2/\rho_2(\pi_1S)$. If $\rho_1$ and $\rho_2$ are not conjugate in $\PSL(2,\mathbb{R})$, Schwartz--Sharp
\cite{SchwartzSharp93}
proved that there exist  two
constants $C = C(\epsilon)>0$ and $\alpha \in (0,1)$ independent of $\epsilon$ such that, for all $\epsilon>0$, we have
\begin{equation}\label{eq:intro-2}
{\rm Card}\{[\gamma] : l_1([\gamma]),l_2([\gamma]) \in (T,T+\epsilon) \} \sim C\frac{e^{\alpha T}}{T^{3/2}} 
\quad \text{ as } 
\quad T \to \infty.
\end{equation}
The number $\alpha$ is called the {\it correlation number}
of $\rho_1$ and $\rho_2$.
Sharp \cite{Sharp98}
showed that $\alpha = a+b$ where $(a,b)$ is the point on the curve $\mathcal{C}(\rho_1,\rho_2)$ for which the tangent line has slope $-1$.

\medskip

Sullivan's dictionary is a set of correspondences,
most often via the lens of
conformal dynamics, between Kleinian groups and rational maps.
Inspired by Ahlfors’ finiteness theorem for Kleinian groups \cite{Ahlfors64}, Sullivan proved the no-wandering domain theorem for rational maps,
which completed the classification of the Fatou components \cite{Sullivan85}. Ever since, the analogy between Kleinian groups and rational maps has been extensively studied, see for instance
\cite{Luo_GeoFiniteDegenI,LLM24,McM_RibbonRtreeHoloDyn,McM_CompExpCircle}.

Via Sullivan's dictionary, McMullen
\cite{McMullen00}
introduced a {\it Poincar\'e series}
$P_f(s,x)$
for every
rational map $f$ on the Riemann sphere $\mathbb{P}^1= \mathbb C\mathbb{P}^1$ of degree $d\geq 2$ by
$$P_f(s,x) \defeq \sum_{n=1}^\infty \sum_{f^n(y)=x}|(f^n)'(y)|^{-s}.$$
The {\it critical exponent $\delta_f$} is defined as the supremum of those $s \ge 0$ such that $P_f(s,x) = \infty$ for all $x \in \bbP^1$. McMullen proved that if $f$ is a geometrically finite rational map (i.e., the critical points in the Julia sets are preperiodic), then $\delta_f$ equals the Hausdorff dimension ${\rm HD}(J(f))$
of the Julia set $J(f)$ and the Poincar\'e series $P_f(s,x)$ is divergent when $s = \delta_f$ for all $x \in \bbP^1$; see \cite[Theorem 1.2]{McMullen00}.

\medskip

In this paper, we define the Manhattan curve for a pair of rational maps and more generally a pair of
holomorphic endomorphisms of $\mathbb{P}^k = \bbC\bbP^k$
in the same hyperbolic component and study its regularity properties. In particular, we show that it is real-analytic; see Theorem \ref{thm_curve_ana}. 
As we will see, 
for
$k\geq 2$,
a central object in this study will be  
the {\it volume dimension} of hyperbolic Julia sets, which we introduced in our earlier paper \cite{BHMane} as an appropriate replacement of the Hausdorff dimension in higher dimensional expanding holomorphic dynamical systems. 

We also discuss several counting results for the multiplier spectrum in the spirit of \eqref{eq:intro-1} and \eqref{eq:intro-2}
and in particular we show that, under suitable assumptions, the Manhattan curve for two holomorphic endomorphisms is related to the correlation number of their multiplier spectra; see Theorem \ref{thm_correlation_num}. This result is new even in dimension $1$.

Manhattan curves are also
related to the so-called {\it pressure metrics}, which are constructed and studied on stable components of various families of rational maps and polynomials \cite{McMullen08,HeNie23}. We discuss this
relationship and other generalizations of Theorem \ref{thm_curve_ana} in Section \ref{sec_pressure}.

\subsection*{Acknowledgements}
The authors would like to thank the Institute for Advanced Study for their support and hospitality
through the Summer Collaborators Program.
This project has received funding from
 the
 Programme
 Investissement d'Avenir
(ANR QuaSiDy /ANR-21-CE40-0016,
ANR PADAWAN /ANR-21-CE40-0012-01,
ANR TIGerS /ANR-24-CE40-3604),
 from the MIUR Excellence Department Project awarded to the Department of Mathematics of the University of Pisa, CUP I57G22000700001,
 and
 from the PHC Galileo project G24-123.
The first
author is affiliated to the GNSAGA group of INdAM.

\section{Manhattan curves in complex dynamics}

\subsection{Definitions and statement}
For every 
$k\ge 1$ and $d\ge 2$, 
denote by ${\rm End}(k,d)$ the space of holomorphic endomorphisms of $\bbP^k$ of algebraic degree $d$. 
A {\it hyperbolic component} of ${\rm End}(k,d)$ is a connected component of the open subset given by \emph{hyperbolic} maps. Recall that a holomorphic endomorphism is {\it hyperbolic}
(or \emph{uniformly expanding on its Julia set})
if there exist $\eta > 1$ and $C>0$ such that
$||Df_x^n(v)|| > C\eta^n||v||$ for every $x \in J(f)$, $v \in T_x \mathbb P^k$, and $n\in\mathbb N$, where we denote by 
$J(f)$ the Julia set of $f$, i.e., the support of its
unique measure of maximal entropy \cite{FLM83,Lyubich82, Lyu83entropy,BD01,DS10}.

Let $\Omega\subset {\rm End} (k,d)$
be a hyperbolic component. Then for any $f_1,f_2 \in \Omega$, the dynamical systems $(J(f_1),f_1)$ and $(J(f_2),f_2)$ are conjugate by a H\"older continuous map. 
Observe that, while when $k= 1$ a stable component (in the sense of 
\cite{Lyu83typical, MSS83}) 
is automatically a hyperbolic component as soon as it contains one hyperbolic parameter, it is an open question whether this 
is the case for $k \geq 2$ (where stability is defined as in \cite{BBD18,B19}), 
apart from special families \cite{AB23}.
For $k\geq 2$
it is also not known whether there always exists a holomorphic motion for the Julia sets on any stable component, as is the case for $k=1$, see \cite{BBD18,BR24} for a weaker (measurable)
version of the holomorphic motion.

Fix $f_0 \in \Omega$ and denote by $J_0=J(f_0)$ the Julia set of $f_0$. Take $f_1,f_2 \in \Omega$ and denote by $h_i \colon J_0 \to J(f_i)$, $i=1,2$ the H\"older conjugacy.

\begin{defn}
For $f_1,f_2 \in \Omega$, $(a,b)\in\mathbb{R}_{\ge 0} \times \mathbb{R}_{\ge 0} \setminus \{(0,0)\}$,
$s \in \mathbb{R}$ and $x \in \mathbb{P}^k$, we define the formal {\it weighted Poincar\'e series} for $f_1$ and $f_2$ as
\begin{equation*}
P_{f_1,f_2}^{a,b}(s,x) \defeq \sum_{n=1}^{\infty} \sum_{f_0^n(y)=x}e^{-s(a\log|\Jac f_1^n(h_1(y))|+b\log|\Jac f_2^n(h_2(y))|)},
\end{equation*}
where we denote by $\Jac f$ 
the complex Jacobian of a holomorphic endomorphism $f$.
\end{defn}

\begin{defn}
We define the {\it critical exponent} $\delta_{f_1,f_2}^{a,b}$ to be the supremum of those $s \ge 0$ such that $P_{f_1,f_2}^{a,b}(s,x) =\infty$ for all $x \in \bbP^k$.
\end{defn}

\begin{defn}
The {\it Manhattan curve} $\mathcal{C}=\mathcal{C}(f_1,f_2)$ of $f_1,f_2$ is defined as
$$\{(a,b)\in\mathbb{R}_{\ge 0} \times \mathbb{R}_{\ge 0} \setminus \{(0,0)\}: \delta_{f_1,f_2}^{a,b} = 1\}.$$
\end{defn}

Before
stating
the main theorem of this section, we give a brief overview of the notion of {\it volume dimension} of a hyperbolic Julia sets,
that we introduced in our earlier paper \cite{BHMane}.
In complex dimension 1, 
the Hausdorff dimension of $J(f)$
is defined by means of covers of balls.
Thanks to the properties of holomorphic maps, and in particular
to the Koebe distortion theorem and related distortion results,
there is a natural interplay
between the Hausdorff dimension and the dynamics of a rational map. In particular, one can recover the Hausdorff dimension (which is a priori just a geometric quantity) from the dynamics, and in particular from the entropy and the Lyapunov exponents of the invariant measures of $J(f)$, see for instance
\cite{manning1984dimension,Mane}.

When $k\geq 2$,
a holomorphic endomorphism of $\mathbb{P}^k$
is no longer conformal.
As a consequence, 
the relation between the Hausdorff dimension 
and the dynamics is less clear. 
In \cite{BHMane}, by means of covers of Bowen balls and exploiting the distortion estimates in \cite{BDM,BD19},
we introduced
a dynamical
\emph{volume dimension}
for expanding sets and measures as an appropriate replacement for the Hausdorff dimension. This dimension
depends on 
the dynamics of $f$ to incorporate
the non-conformality
of holomorphic endomorphisms for $k\ge2$.
Several characterizations of the volume dimensions
are given in \cite{BHMane}, as generalizations to any dimensions of characterizations in 
\cite{DU1,DU2, McMullen00, PU}. 
In this paper, we will only need the following one.

\begin{thm}[\cite{BHMane}] \label{thm_bh}
Let $f$ be a hyperbolic endomorphism in ${\rm End} (k,d)$, $d\geq 2$. Then
twice
the volume dimension 
$\VD_f(J(f))$ of the Julia set $J(f)$ 
equals the first zero of the pressure function $P_J^+(t) \defeq \sup\{h_\nu(f)-tL_\nu\}$ where the supremum is taken over all the invariant probability measures supported on $J(f)$ and $L_\nu$ denotes the sum of (complex) Lyapunov exponents of $\nu$.

\end{thm}

In particular, our volume dimension is equivalent to half of the Hausdorff dimension when $k=1$.

\medskip

The main result of the section
is the following theorem. We say that $f_1$ and $f_2$ have {\it proportional} marked length spectrum if there exists a constant $c>0$ such that $\widetilde{\lambda}(h_1(\hat{z})) = c\cdot \widetilde{\lambda}(h_2(\hat{z}))$ for any primitive cycle $\hat{z}$ for $f_0$. Here we denote $\widetilde{\lambda}(\hat{z}) = \log |\Jac f_0^n(z)|$ with $\hat{z}\defeq \{z,\ldots, f_0^{n-1}(z)\}$.
\begin{thm}\label{thm_curve_ana}
Let $\Omega$ be a 
hyperbolic component of ${\rm End}(k,d)$.
Take 
$f_1,f_2 \in \Omega$. 
Then the Manhattan curve $\mathcal{C}(f_1,f_2)$ satisfies the following properties:
\begin{enumerate}
\item the points $(2\VD_{f_1}(J(f_1)),0)$ and $(0,2\VD_{f_2} (J(f_2)))$ belong to
$\mathcal{C}(f_1,f_2)$;
\item $\mathcal{C}(f_1,f_2)$ is real-analytic;
\item $\mathcal{C}(f_1,f_2)$ is convex;
\item $\mathcal{C}(f_1,f_2)$ is a straight line if and only if the marked length spectra of $f_1$ and $f_2$ are proportional.
\end{enumerate}
\end{thm}

We will prove Theorem \ref{thm_curve_ana} in the next section and 
we will discuss some variations in Section
\ref{sec_pressure}.

\subsection{Proof of Theorem \ref{thm_curve_ana}}\label{sec_pf} 
Let $\Omega$ be a hyperbolic component of ${\rm End}(k,d)$ and $f_0$ a fixed element in $\Omega$, as above. 
Define two functions $\tau_i \colon J_0 \to \mathbb{R}$, $i=1,2$
by
$$\tau_i(x) \defeq \log |\Jac f_i(h_i(x))|,$$
where $h_i$ is
the conjugating map between $f_0$ and $f_i$ 
on their Julia sets.
We note that the $\tau_i$'s
are H\"older continuous.

Recall that for a continuous function $\eta \colon J_0 \to \mathbb{R}$, the {\it topological pressure} $\mathcal{P}(\eta)$ is defined as
$$\mathcal{P}(\eta) \defeq \sup_{\mu} \Big(
h_\mu(f_0)+\int_{J_0}\eta\, d\mu\Big),$$
where the supremum is taken over all the $f_0$-invariant probability measures on $J_0$.

Let 
$C^\kappa(J_0)$ be the space of $\kappa$-H\"older continuous functions on $J_0$. Given $\phi \in C^\kappa(J_0)$, define the {\it transfer operator} $\mathcal{L}_\phi \colon C^\kappa(J_0) \to C^\kappa(J_0)$ by
$$\mathcal{L}_\phi\eta(x) \defeq \sum_{f_0(y)=x}e^{\phi(y)}\eta(y).$$
Since $(J_0,f_0)$ is uniformly hyperbolic, the transfer operator is quasi-compact
\cite{LyubichY-introduction}; 
namely, it has an isolated maximal
real eigenvalue $\lambda_\phi$
of multiplicity $1$
and the rest of the spectrum is contained in a disk of radius strictly smaller than $\lambda_\phi$. Moreover, we have $\lambda_\phi = e^{\mathcal{P}(\phi)}$. Standard references are \cite{Parry90, Ruellebook,PU}.

\medskip

The following Bowen-type theorem is a key ingredient in proving that the Manhattan curve is real-analytic.

\begin{thm}\label{thm_Bowen}
For every
$(a,b)\in\mathbb{R}_{\ge 0} \times \mathbb{R}_{\ge 0} \setminus \{(0,0)\}$
we have
$\mathcal{P}(-\delta^{a,b}_{f_1,f_2}(a\tau_1 +b\tau_2))=0$. In particular, the Manhattan curve $\mathcal{C}(f_1,f_2)$ is the set of the solutions to the equation
$\mathcal{P}(-a\tau_1-b\tau_2)=0$.
\end{thm}
\begin{proof}
For $s \ge 0$, write $\phi_s \defeq -s(a\tau_1+b\tau_2)$.
Then,
by the definition of
the transfer operator,
we have
\begin{equation}\label{eq_1}
P_{f_1,f_2}^{a,b}(s,x) = \sum_{n\ge 1} \mathcal{L}^n_{\phi_s}\mathds{1}(x),
\end{equation}
where $\mathds{1}$ denotes the function constantly equal to $1$.
We denote by $\lambda_s=e^{\mathcal{P}(\phi_s)}$ the maximal
eigenvalue of $\mathcal{L}_{\phi_s}$.

By definition of critical exponent and \eqref{eq_1}, 
$\sum_{n\ge 1} \mathcal{L}^n_{\phi_s}\mathds{1}$ 
is divergent for all $x \in \mathbb{P}^k$ if $s < \delta_{f_1,f_2}^{a,b}$. Since,
for large $n$,  we have
$\mathcal{L}^n_{\phi_s}\mathds{1}(x) \sim \lambda_s^n(x)$ for any $x\in \mathbb{P}^k$, we 
deduce the equality
$\lambda_{\delta_{f_1,f_2}^{a,b}} = 1$. 
It follows that $\mathcal{P}(\phi_{\delta_{f_1,f_2}^{a,b}}) = 0$; namely, $\mathcal{P}(-\delta_{f_1,f_2}^{a,b}(a\tau_1+b\tau_2)) = 0$, which gives the first assertion. 

Since $\mathcal{C}(f_1,f_2)$ is by definition the collection of $(a,b)$ such that $\delta_{f_1,f_2}^{a,b}=1$, the second assertion follows.
This completes the proof.
\end{proof}

Now we give a proof of Theorem \ref{thm_curve_ana}.

\begin{proof}[Proof of Theorem \ref{thm_curve_ana}]
We prove statement (1).
If $a=0$, by Theorem \ref{thm_Bowen}, we have $\mathcal{P}(-b\tau_2) = 0$.
This implies 
the equality
$b = 2\VD_{f_2}(J(f_2))$ by Theroem \ref{thm_bh}.
Similarly, 
the point $(2\VD_{f_1}(J(f_1)),0)$ is also
on $\mathcal{C}(f_1,f_2)$. 

\medskip

Since $f_0 \colon J_0 \to J_0$ is uniformly hyperbolic, the pressure function $\mathcal{P}$ is analytic on the space of H\"older continuous functions on $J_0$
\cite{Parry90, Ruellebook,PU}.
Therefore, statement (2)
follows from
Theorem \ref{thm_Bowen}
and 
the inverse function theorem.

\medskip

Statement (3) follows from the fact that, for all $x$, the set
$$S_x
\defeq \{(a,b) \in \mathbb{R}_{\ge 0} \times \mathbb{R}_{\ge 0} \setminus \{(0,0)\} : P_{f_1,f_2}^{a,b}(1,x) < \infty\}$$ is convex. The convexity of $S_x$
follows from 
H\"older's inequality, i.e.,
$$P^{ta_1+(1-t)b_1,ta_2+(1-t)b_2}_{f_1,f_2}(1,x) \le \left(P^{a_1,b_1}_{f_1,f_2}(1,x)\right)^{t} \cdot \left(P^{a_2,b_2}_{f_1,f_2}(1,x)\right)^{1-t}.$$

\medskip

To prove statement (4), if $f_1$ and $f_2$ have
proportional marked length spectrum, then it is straightforward to see that $\mathcal{C}(f_1,f_2)$ is a straight line. 
For the converse, suppose $\mathcal{C}(f_1,f_2)$ is a straight line. 
Since $(2{\rm VD}_{f_1}(J(f_1)),0)$ and $(0,2{\rm VD}_{f_2}(J(f_2)))$ 
belong to $\mathcal{C}(f_1,f_2)$,
its equation is given by
$$b(a)
= 2{\rm VD}_{f_2}(J(f_2))) - \frac{2{\rm VD}_{f_2}(J(f_2)))}{2{\rm VD}_{f_1}(J(f_1)))}a.$$
Since $\frac{d\mathcal{P}(-a\tau_1-b(a) \tau_2)}{da} = 0$, spelling out this derivative and plugging in the above expression for $b(a)$ 
gives
\begin{equation}\label{eq_claim0}
\frac{2{\rm VD}_{f_2}(J(f_2))}{2{\rm VD}_{f_1}(J(f_1))} = \frac{\int_{J_0} \tau_1 dm_2}{\int_{J_0} \tau_2 dm_2},
\end{equation}
where $m_i$ is the equilibrium state of $-2{\rm VD}_{f_i}(J(f_i))\tau_i
=-2{\rm VD}_{f_i}(J(f_i))\log|\Jac f_i\circ h_i|\colon J_0 \to \mathbb{R}$.

\medskip

{\bf
Claim.} The functions
$-2{\rm VD}_{f_1}(J(f_1))
\tau_1$
and $-2{\rm VD}_{f_2}(J(f_2))
\tau_2$ 
are cohomologous.

\smallskip

\begin{proof}[Proof of the claim.]
We first note that,
since $\mathcal{P}(-2{\rm VD}_{f_2}(J(f_2))\cdot \tau_2) = 0$, we
have
\begin{equation} \label{eq_claim1}
h(m_2)=2{\rm VD}_{f_2}(J(f_2))\int \tau_2 dm_2.
\end{equation}
We now
show that $m_2$ is also an equilibrium state for $-2{\rm VD}_{f_1}(J(f_1))\tau_1$. 
To this end, we compute
\begin{align*}
h(m_2)-2{\rm VD}_{f_1}(J(f_1))\int \tau_1 dm_2
&= 2{\rm VD}_{f_2}(J(f_2))\int \tau_2 dm_2 -2{\rm VD}_{f_1}(J(f_1))\int \tau_1 dm_2 \text{ by } \eqref{eq_claim1}\\
&=0 \text{ by } \eqref{eq_claim0}\\
& = \mathcal{P}(-
2{\rm VD}_{f_1}(J(f_1))
\tau_1).
\end{align*}
Since the equilibrium state of $-
2{\rm VD}_{f_1}(J(f_1))
\tau_1$ is unique, we have $m_1=m_2$. This implies that $-
2{\rm VD}_{f_1}(J(f_1))\tau_1$ and 
$-
2{\rm VD}_{f_2}(J(f_2))
\tau_2$ are cohomologous, see for instance
\cite[Theorem 4.8]{Sarig09}. 
\end{proof}

The Claim implies that $f_1$ and $f_2$ have
proportional marked length spectra, and completes the proof.
\end{proof}

\subsection{Generalizations and variations of Theorem \ref{thm_curve_ana}, and relation with the pressure metrics} \label{sec_pressure}
While stated 
for hyperbolic maps for simplicity,
Theorem \ref{thm_curve_ana} actually holds in larger generality, at least in dimension 1.
Indeed, recall that 
McMullen \cite{McMullen00} defined and studied the Poincar\'e series $P_f(s,x)$ for every 
{\it geometrically finite} rational map. Similarly, we can define and study Manhattan curves for a pair of rational maps that are not necessarily hyperbolic. 

As a first example, one can consider 
Misiurewicz maps, or more precisely, a stable 
component in a family of rational maps given by a Misiurewicz relation where all critical points are either preperiodic to a repelling cycles or in the basin of an attracting cycle.
In this case, the analyticity statement 
in Theorem \ref{thm_curve_ana} essentially follows from our
study
in \cite{BH25} (which is based on the spectral study in \cite{MakSmir03},
see in particular \cite[Theorem 3.1 and Remark 3.3]{BH25}). 

One can also consider similar components
inside parabolic families, i.e., families with some persistent parabolic cycle.
In this case,
a variation of the Manhattan curve (where 
the condition $\delta^{a,b}_{f_1, f_2}=1$
is replaced by $\delta^{a,b}_{f_1, f_2}=e^\eta$ for some $0<\eta<\log d$)
can be studied with the techniques as in \cite{BH24}, and in particular with the machinery of \cite{BD23eq1,BD24eq2}, which gives a suitable Banach space for the spectral study of the transfer operators.

As mentioned in the Introduction, Manhattan curves are
also related to pressure metrics
\cite{McMullen08,Ivrii14,HeNie23,BH24,BH25,HLP23,HLP25}.
In the context of hyperbolic surfaces (possibly with cusps), such a relation was obtained by Pollicott-Sharp \cite[Section 4]{PS16} and Kao \cite[Theorem D]{Kao20}. 
Using similar arguments, we can
obtain analogous results in holomorphic dynamics - if $\{f_t\}$ is a smooth path in a stable component, then the family of Manhattan curves $\mathcal{C}(f_0,f_t)$ encodes the information of the 
pressure metric $\|\frac{d}{dt}|_{t=0}f_t\|_p$. More precisely, if $(a,b_t(a))$ is a parametrization of
the Manhattan curve $\mathcal{C}(f_0,f_t)$, then we have $$\frac{d^2}{dt^2}\bigg|_{t=0}b_t(a) = a(a-2\VD_{f_1}(J(f_1))) \cdot \Big\|\frac{d}{dt}\big|_{t=0}f_t\Big\|_p^2$$
for every $a \in (0,2\VD_{f_1}(J(f_1)))$.

\section{Counting problems for the multiplier spectrum}
In this section, we discuss counting problems
for the multiplier spectrum
in complex dynamics. As in the case
of the length spectrum
of hyperbolic surfaces, we will see that,
for two holomorphic endomorphisms $f_1,f_2$ in the same hyperbolic component of ${\rm End} (k,d)$, 
the Manhattan curve $\mathcal{C}(f_1,f_2)$ is related to the correlation number of their 
multiplier spectra.

\subsection{Asymptotic distribution of multipliers}
For $f \in {\rm End}(d,k)$, we denote by $\mathcal{P}(f)$ the set of primitive periodic orbits $\hat{z} \defeq \{z,\ldots,f^{n-1}(z)\}$ of $f$ on the Julia set $J(f)$. For each $\hat{z} \in \mathcal{P}(f)$, denote by 
$$\lambda(\hat{z}) \defeq  \Jac f^n(z)$$
the {\it multiplier} of the periodic orbit $\hat{z}$.
For every $T\geq 0$, 
consider the quantity
$$N_T(f) \defeq {\rm Card}\{\hat{z} \in \mathcal{P}(f) : |\lambda(\hat{z})| \le T\}.$$
As 
$f$ is hyperbolic, the set $N_T(f)$ is finite for each $T > 1$.

As mentioned in the Introduction, for hyperbolic surfaces $\mathbb{H}^2/\rho_1(\pi_1S)$, Pollicott-Sharp \cite{PS98} studied the asymptotics for the cardinality ${\rm Card} \{[\gamma] : l_1([\gamma]) \le T \}$ of conjugacy classes $[\gamma]$ in $\pi_1S$ such that the primitive geodesic in the class $[\gamma]$ has hyperbolic length smaller than or equal to $T$. 
The following theorem gives the main term in the asymptotics 
of the analogous function
$N_T(f)$.
In particular, this shows that the asymptotic growth rate of the modulus of multipliers is again
related to
the volume dimension
of the Julia set. Recall that, when $k=1$, the volume dimension is equal to half 
of the Hausdorff dimension.

\begin{thm}\label{thm_higher}
Let $f \colon \mathbb{P}^k \to \mathbb{P}^k$ be a hyperbolic holomorphic endomorphism of algebraic degree $d \ge 2$ such that $\log |\Jac f|$ is not cohomologous to a constant. Then we have
\[N_T(f) \sim Li(T^{\rm 2VD})
\quad \mbox{ where }
\quad
Li(x) = \int_0^t \frac{dt}{\ln t}\] 
and ${\rm VD}$ is the volume dimension of the Julia set of $f$.
\end{thm}

\begin{rmk}
In dimension 1, if $f$ is a hyperbolic rational map which is not conjugate to $z^{\pm d}$, Oh-Winter \cite{Oh17} proved that there exists $\epsilon>0$ such that
$N_T(f) = Li(T^{\HD (J)})+O(T^{\HD(J)-\epsilon})$. Then
 it is a natural question
whether one can find a similar power saving error term in Theorem \ref{thm_higher} when $k\ge2$, as well as
generalize such results 
without the uniform hyperbolicity assumption, see for instance 
\cite{LRL}. 
A power saving error term in Theorem \ref{thm_higher} would allow one to also
study the {\it shrinking interval} problem for the asymptotic distribution of arguments of the multipliers; see \cite{HeNie22} for the case $k=1$.
\end{rmk}

\begin{proof}[Proof of Theorem \ref{thm_higher}]
For every $s \in \mathbb C$,
let us consider the (formal)
\emph{Ruelle zeta function}
$$\eta(s) \defeq \exp \left(\sum_{n=1}^\infty \frac{1}{n} \sum_{f^n(z) = z} e^{-s\cdot |\lambda(\hat{z})|} \right).$$
We are going to show that $\eta(s)$ satisfies the following three properties:
\begin{enumerate}
\item it converges for ${\rm Re}(s) > 2\VD$;
\item it has a simple pole at $s = 2\VD$; and 
\item it is analytic on ${\rm Re}(s) \ge 2\VD$ except at $s = 2\VD$
(meaning that every point in this set has a neighborhood where $\eta(s)$ is analytic).
\end{enumerate}
Once the
properties above are established, the conclusion follows from the Ikehara’s Tauberian theorem (see for instance \cite{Lang94}) as $\eta(s) = \int_0^\infty e^{-st}dN_t$. See \cite{PS97Circle} for a similar application.

\medskip

Since $f$ is hyperbolic,
Ruelle's work \cite{Ruelle90} 
relates the zeta function to the transfer operator $\mathcal{L}_s$ for the potential $-s\log |\Jac f|$ (see also \cite[Section 6]{Oh17}). Hence, we know that $\eta(s)$ converges and is analytic when the spectral radius $e^{P(-s\log|\Jac f|)}$ of $\mathcal{L}_s$ is strictly smaller than 1, where $t\mapsto P(-t\log|\Jac f|)$ is the pressure function for the potential $-t\log |\Jac f|$. 
In particular, if ${\rm Re}(s) > p(f)$, where $p(f)$ is the unique zero of $P(-s\log|\Jac f|)$, $\eta(s)$ converges and is analytic. Moreover, by
Theorem \ref{thm_bh},
we know that $p(f) = 2\VD$.
This shows the property (1).

\medskip

Write $s = a +bi$. We show that $\eta(s)$ has a pole at $s = 2\VD + bi$ if and only if $b =0$. By \cite[Proposition 3]{PS97Circle},
$s = 2\VD + bi$ is a pole if and only if $P(-(2\VD+bi)\log |\Jac f|) = 0$, if and only if
\begin{equation} \label{eq_2.21}
-b\log |\Jac f| = u \circ f - u +\psi
\end{equation}
for some continuous functions $u \colon J(f) \to \mathbb R$ and $\psi \colon J(f) \to 2\pi \mathbb{Z}$. 
Suppose $s = 2\VD + bi, b\neq 0$ is a pole.
Since by assumption 
$\log |\Jac f|$ is not cohomologous to a constant,
a simple argument 
(see for instance \cite[Lemma 5.14]{BD24eq2})
shows that we have
\[P(-(2\VD+bi)\log |\Jac f|) < 0,\]
as the eigenvalue $e^{P(-(2\VD+bi)\log |\Jac f|)}$ of the transfer operator is strictly smaller than 1.
This shows that $s = 2\VD + bi$ cannot be a pole and therefore is a contradiction. Hence we must have $b=0$. By analyticity, the same argument also shows (3).
\medskip

To see that the pole $s = 2\VD$ is simple, we compute
$$\lim_{s \to 2\VD} \frac{s-2\VD}{1-e^{P(-s\log|\Jac f|)}} = \frac{-1}{P'(-2\VD \cdot \log |\Jac f|)},$$
which is strictly
positive
as the pressure function $t\mapsto P(-t\log |\Jac f|)$ is strictly decreasing.
This completes the proof of
(2), and concludes the proof.
\end{proof}

\subsection{Correlation of multiplier spectra}
As mentioned in the  Introduction, in the context of hyperbolic geometry, the Manhattan curve $\mathcal{C}(\rho_1,\rho_2)$ is related to the correlation of the length
spectra of
the hyperbolic surfaces $\mathbb{H}^2/\rho_1(\pi_1S)$ and $\mathbb{H}^2/\rho_2(\pi_1S)$. 
Analogously, if $f_1$ and $f_2$ are two holomorphic endomorphisms of $\mathbb{P}^k$ such that $(J(f_1),f_1)$ and $(J(f_2),f_2)$  are H\"older conjugate by $\phi$, 
for every $T,\epsilon>0$
we 
can consider the
quantity
$$M_{T,\epsilon}(f_1,f_2) \defeq {\rm Card}\{\hat{z} \in \mathcal{P}(f_1): \widetilde{\lambda}_1(\hat{z}),\widetilde{\lambda}_2(\hat{z}) \in (T,T+\epsilon) \}$$
where $\hat{z} = \{z,\ldots, f^{n-1}(z) \}$ is a primitive periodic orbit of $f_1$
and we set
$\widetilde{\lambda}_1(\hat{z})=\log|\Jac f_1^n(z)|$ and $\widetilde{\lambda}_2(\hat{z})=\log|\Jac f_2^n(\phi(z))|$.
Observe that $M_{T,\epsilon}(f,f)$ is essentially
a  ``discrete derivative'' of $N_T(f)$, up to taking logarithm.

\begin{thm}\label{thm_correlation_num}
Let $\Omega \subset  {\rm End}(d,k)$ be a hyperbolic component and 
take $f_1,f_2 \in \Omega$.
Let $\phi$ be the conjugating map
between $f_1$ and $f_2$
and suppose that
\begin{enumerate} [label=(\alph*)]
\item $b_1\log |\Jac f_1|+b_2\log |\Jac f_2  \circ \phi|$
is not cohomologous to a constant for any $(b_1,b_2) \in \mathbb{R}^2 \setminus \{(0,0)\}$;
\item there exist two periodic orbits $\hat z$ and $\hat w$
for $f_1$
with
$\tilde \lambda_1  (\hat z) > 
\tilde \lambda_2  (\hat z)$ and
$\tilde \lambda_2  (\hat w) > 
\tilde \lambda_1  (\hat w)$.
\end{enumerate}
Then the following properties hold.
\begin{enumerate}
\item For every
$\epsilon>0$
there exist two constants
$C = C(\epsilon,f_1,f_2)>0$ 
and $\alpha = \alpha(f_1,f_2) \in (0,k\log d)$
(independent of $\epsilon$)
such that
$$
M_{T,\epsilon} (f_1, f_2)
\sim C \frac{e^{\alpha T}}{T^{3/2}}
\quad \text{ as }
\quad T \to \infty.$$
\item The exponent $\alpha$ in (1) equals $a_0+b_0$, where $(a_0,b_0)$
 is the unique point on 
$\in \mathcal{C}(f_1,f_2)$
such that the slope of the tangent line
to $\in \mathcal{C}(f_1,f_2)$
at $(a_0,b_0)$ is equal to $-1$. 
\end{enumerate}
\end{thm}

We will prove the two assertions in Theorem 
\ref{thm_correlation_num} by adapting
the arguments
in \cite[Section 3]{SchwartzSharp93} and 
\cite[Section 3]{Sharp98}, respectively.
In particular, we will use a result of Lalley \cite{Lalley87} concerning the distribution of periodic orbits of a symbolic flow. In our case, we consider the following suspension flow for $f_1$. 
Consider first the natural extension
$(S_{f_1},\hat{f_1})$ of $(J(f_1), f_1)$,
where 
$$S_{f_1} \defeq \{(\ldots, x_{-2}, x_{-1},x_0):x_0 \in J(f_1)\}$$ and
$\hat{f_1}(\ldots, x_{-2}, x_{-1},x_0) \defeq
(\ldots, x_{-2}, x_{-1},x_0, f_1(x_0))$.
We then consider the flow associated to $(S_{f_1},\hat{f_1})$ and the 
roof function $r = \log |\Jac f_1|$. Namely, we consider the space
\[
S_{f_1}^{(r)} \defeq \{
(\hat{x}, s)\colon
\hat{x} \in S_{f_1}, 0\leq s\leq r(x_0) \}/(\hat x, r(x_0))\sim (\hat f_1(\hat x), 0).
\]
and the natural semigroup $\psi_t$ 
on
$S_{f_1}^{(r)}$ given by
by $\psi_t (\hat x, s)= (\hat x, s+t)$. Observe that each $n$-periodic orbit $(x_0, \dots, x_{n-1})$
for $f_1$ corresponds to a closed curve for this flow, of length
$r(x_0)+ \dots + r(f^{n-1}(x_0)) = \log |f_1^n (x_0)|$.

Lalley's result concerns H\"older continuous functions on the symbolic flow which are {\it independent}. In our case, this condition
will be satisfied 
thanks to the following lemma.

\medskip

\begin{lem} \label{lem_sharp}
Let $f_1$ and $f_2$ be as in Theorem \ref{thm_correlation_num}.
If there exist $a_1,a_2 \in \mathbb{R}$ such that 
\begin{equation}\label{eq:lem_sharp}
a_1\widetilde{\lambda}_1(\hat{z})+a_2\widetilde{\lambda}_2(\hat{z}) \in \mathbb{Z}
\quad \mbox{ for every }
\hat{z} \in \mathcal{P}(f_1),
\end{equation}
then $a_1=a_2=0$.
\end{lem} 
\begin{proof}
Suppose by contradiction that 
there exist $a_1, a_2$ with $a_1^2+a_2^2 \neq 0$ such that \eqref{eq:lem_sharp} holds.
Set $\zeta \defeq a_1\log |\Jac f_1| +a_2\log |\Jac f_2 \circ \phi| \colon J(f_1) \to \mathbb{R}$. 
We note that $\zeta$ is not cohomologous to a constant by the
assumption (a) in Theorem \ref{thm_correlation_num}. Since 
we have
$a_1\widetilde{\lambda}_1(\hat{z})+a_2\widetilde{\lambda}_2(\hat{z}) \in \mathbb{Z}$ for every $\hat{z} \in \mathcal{P}(f_1)$, by Livsic theorem
\cite{Livsic71},
there exist $b \in \mathbb{R} \setminus \{0\}$, a
continuous functions $u \colon J(f_1) \to \mathbb{R}$, and a function
$\psi \colon J(f_1) \to 2\pi \mathbb{Z}$ such that
$$-b\zeta = u \circ f - u + \psi.$$
This implies that we have
$P(-(\delta(\zeta) + bi)\zeta)=0$,
where $\delta(\zeta)$ is the unique
real number such that $P(-\delta(\zeta) \zeta)=0$. However, this contradicts the fact that 
we must have $P(-(\delta(\zeta) + bi)\zeta)<0$ 
as
$\zeta$ is not cohomologous to a constant
(see, e.g., \cite[Lemma 5.14]{BD24eq2}). Therefore, we must have $a_1 = a_2 = 0$, as desired.
\end{proof}

We can now give a proof of Theorem \ref{thm_correlation_num}. As the arguments are essentially the same as in \cite{SchwartzSharp93} and \cite{Sharp98}, we will only sketch some parts.

\begin{proof}[Proof of Theorem \ref{thm_correlation_num}]
First of all, one can
construct
a strictly positive H\"older continuous function $\psi \colon J(f_1) \to \mathbb{R}$ such that
\begin{equation}\label{eq:psi-orbits}
\int_{\ell} \psi = \widetilde{\lambda}_2(\hat{z})
\end{equation}
if $\ell$ 
is a periodic orbit of the suspension flow corresponding to a periodic orbit $\hat{z}$.
Moreover, $\psi$ is not cohomologous to a constant.
This implies that the function $P(t\psi)$ is stricty convex, and in particular that 
its
derivative is a strictly increasing function.

\medskip

Thanks to Lemma \ref{lem_sharp},
Lalley's theorem \cite{Lalley87} implies that
(1) holds,
as soon as we can verify that $1$ is a value of the derivative of the maps $t\mapsto P(t\psi)$
(say taken at a unique $t=t_1$, by the monotonicity of the derivative)
and that
the entropy $h$
of the equilibrium state associated to 
$t_1\psi$ satisfies $0<h< k\log d$ 
(observe that $\alpha$ will be then equal to this entropy $h$). 
The proof of this fact
is as in \cite{SchwartzSharp93}
and uses the
assumption (b)
on the
existence of the two cycles $\hat z$ and $\hat w$ with different multipliers.

\bigskip

In order to prove
(2), we will follow
the argument in
\cite[Section 2]{Sharp98}.
Fix $(a,b)\in \mathbb R^2\setminus (0,0)$.
It follows
from \eqref{eq:psi-orbits}
that
the Poincar\'e series
\[\sum_{\hat{z}} e^{-a\widetilde{\lambda}_1(\hat{z}) - b\widetilde{\lambda}_2(\hat{z})} = \sum_{\tau} e^{\int_\tau -a-b\psi}\]
converges whenever $(a,b)$ 
is such that 
 $P(-a-b\psi)<0$
 and diverges when $P(-a-b\psi)>0$, and $\mathcal{C}(f_1,f_2)$ is
equal to
 the set
$$\{(a,b) : P(-a-b\psi)=0\} = \{(a,b) : P(-b\psi)=a\}.$$

We think of $b = b(a)$ as a function of $a$. We are interested in the point on $\mathcal{C}(f_1,f_2)$ with the property $db(a)/da = -1$. For every $(a,b(a))$ on the curve $\mathcal{C}(f_1,f_2)$ (i.e., $a = P(-b(a)\psi)$), taking derivative with respect to $a$ gives
\begin{equation*}
1 = \frac{d}{da}P(-b(a)\psi)
 = -\int \psi d\mu_{-b(a)\psi} \cdot \frac{db(a)}{da}.
\end{equation*}
Therefore, if $db(a)/da =- 1$, we must have
\begin{equation*}
\int \psi d\mu_{-b(a)\psi} = 1.
\end{equation*}

Since $\alpha = h_{\mu_{t_1\psi}}$ where $t_1$
is the point (as in the first part of the proof)
such that $\int \psi d\mu_{t_1\psi} = 1$, we know that $t_1$ is also
the value $-b(a)$ of the point $(a,b)$ with $db(a)/da=1$.
Moreover, since $\mu_{t_1\psi}$
is the
equilibrium state for $t_1\psi$, 
we have
\begin{equation*}
\alpha 
= P(t_1\psi) 
-
\int t_1\psi d\mu_{{t_1\psi}}
= P(t_1\psi) -
t_1
 = a -t_1
 =
 a+b(a).
\end{equation*}
This completes the proof of the theorem.
\end{proof}

\begin{rmk}
As in \cite{SchwartzSharp93}, the constant 
$C$ in Theorem \ref{thm_correlation_num} (1) is related to the second derivative of $P(t\psi)$ at $t=t_1$, and is of order $\sim \epsilon^2$.
\end{rmk}

\printbibliography

\end{document}